\documentclass[12pt]{amsart}
\usepackage{fullpage}
\usepackage{times}
\usepackage{latexsym}
\usepackage{amssymb}
\usepackage[all]{xy}

\newtheorem{thm}{Theorem}[section]

\newtheorem{cor}[thm]{Corollary}

\newtheorem{lem}[thm]{Lemma}

\newtheorem{prob}[thm]{Problem}

\theoremstyle{remark}
\newtheorem{rem}[thm]{Remark}

\usepackage{hyperref}

\newcommand{\CP}{\mathbb{CP}}

\newcommand{\R}{\mathbb{R}}
\newcommand{\Z}{\mathbb{Z}}

\newcommand{\SO}{\mathrm{\SO}}
\newcommand{\SL}{\mathrm{SL}}
\newcommand{\GL}{\mathrm{GL}}

\newcommand\Tor{\operatorname{Tor}}

\title[Anosov diffeomorphisms on Thurston geometric 4-manifolds]{Anosov diffeomorphisms on Thurston geometric 4-manifolds}

\author{Christoforos Neofytidis}
\address{Department of Mathematics, Ohio State University, Columbus, OH 43210, USA}
\email{neofytidis.1@osu.edu}
\date{\today}
\subjclass[2010]{37D20, 55R10, 57M05, 20F34, 57R19, 58A30}
\keywords{Anosov diffeomorphisms, Thurston geometries, 4-manifolds}

\begin{document}

\begin{abstract}
A long-standing conjecture asserts that any Anosov diffeomorphism of a closed manifold is finitely covered by a diffeomorphism which is topologically conjugate to a hyperbolic automorphism of a nilpotent manifold. In this paper, we show that any closed 4-manifold that carries a Thurston geometry and is not finitely covered by a product of two aspherical surfaces does not support (transitive) Anosov diffeomorphisms. 
\end{abstract}

\maketitle

\section{Introduction}

Let $M$ be a closed oriented smooth $n$-dimensional manifold. A diffeomorphism $f\colon M\to M$ is called {\em Anosov} if there exists a $df$-invariant splitting $TM=E^s\oplus E^u$ of the tangent bundle of $M$, together with constants $\mu\in (0,1)$ and $C > 0$, such that for all positive integers $m$
\begin{equation*}
\begin{split}
\|df^m(v)\| \leq C\mu^m\|v\|,\ \text{if} \ v \in E^s,\\  
\|df^{m}(v)\| \leq C^{-1}\mu^{-m}\|v\|,\ \text{if} \ v \in E^u.
\end{split}
\end{equation*}

The invariant distributions $E^s$ and $E^u$ are called the {\em stable} and {\em unstable} distributions. An Anosov diffeomorphism $f$ is said to be of {\em codimension $k$} if $E^s$ or $E^u$ has dimension $k \leq [n/2]$, and it is called {\em transitive} if there exists a point whose orbit is dense in $M$.

One of the most influential conjectures in dynamics, dating back to Anosov and Smale~\cite{Sm}, is that any Anosov diffeomorphism $f$ of a closed manifold $M$ is finitely covered by a diffeomorphism which is topologically conjugate to a hyperbolic automorphism of a nilpotent manifold. In this paper we prove the following:

\begin{thm}\label{t:main}
If $M$ is a closed 4-manifold that carries a Thurston geometry other than $\R^4$, $\mathbb H^2\times\R^2$ or the reducible $\mathbb H^2\times\mathbb H^2$ geometry, then $M$ does not support transitive Anosov diffeomorphisms.
\end{thm}

Some cases have already been studied in arbitrary dimensions, most notably the hyperbolic geometries. In many of the other cases, our proof will rely on certain properties of the fundamental groups of manifolds modeled on specific geometries. We will show existence of a degree one cohomology class $u\in H^1(M;\Z)$ that is fixed under an iterate of any diffeomorphism $f\colon M\to M$. Then we will be able to exclude the possibility for $f$ being Anosov by exploiting Hirsch's study~\cite{Hirsch} on those cohomology classes; cf. Theorems \ref{t:Hirschcov} and \ref{t:Hirsch}. Hirsch's work has already been applied in certain cases, such as on mapping tori of hyperbolic automorphisms of the torus of any dimension or products of such mapping tori with a torus of any dimension~\cite{Hirsch}. In dimension four, these manifolds correspond (up to finite covers) to the geometries $Sol^4_0$, $Sol^4_{m\neq n}$ or $Sol^3\times\R$. Among the most interesting remaining examples include, on the one hand, manifolds with virtually infinite first Betti numbers, such as manifolds modeled on the geometry $\widetilde{SL_2}\times\R$, and, on the other hand, certain polycyclic manifolds; in fact, the case of $Nil^3\times\R$  indicates an error in the proof of~\cite[Theorem 9(a)]{Hirsch}; see Remark \ref{r:Hirsch}.

We should point out that the transitivity assumption in Theorem \ref{t:main} is mild and will only be used when $M$ is virtually an $S^2$-bundle over an aspherical surface $\Sigma_h$, i.e. of genus $h\geq1$. Franks~\cite{Fr} and Newhouse~\cite{Newhouse} proved that a codimension one Anosov diffeomorphism exists only on manifolds which are homeomorphic to tori. It will therefore suffice to examine the existence of codimension two Anosov diffeomorphisms. For a transitive Anosov diffeomorphism $f\colon M\to M$ of codimension $k$, Ruelle-Sullivan \cite{RS} exhibit a cohomology class $\alpha\in H^k(M;\R)$ such that $f^*(\alpha)=\lambda\cdot\alpha$, for some positive $\lambda\neq 1$ (which depends on the topological entropy of $f$). In the light of the latter, we will rule out codimension two transitive Anosov diffeomorphisms on products of type $S^2\times\Sigma_h$, where $h\geq1$. The non-existence of transitive Anosov diffeomorphisms on sphere bundles over surfaces is also part of a more general study of Gogolev-Rodriguez Hertz using cup products~\cite{GH}.

Recall that a manifold modeled on $\R^4$ is finitely covered by the 4-torus and a manifold modeled on the $\mathbb H^2\times\R^2$ geometry or the reducible $\mathbb H^2\times\mathbb H^2$ geometry is finitely covered by the product of the 2-torus with a hyperbolic surface or the product of two hyperbolic surfaces respectively. Thus, Theorem \ref{t:main} excludes transitive Anosov diffeomorphisms on any geometric 4-manifold which is not finitely covered by a product of surfaces $\Sigma_g\times\Sigma_h$, where $g,h\geq 1$. Clearly $T^4=T^2\times T^2$ (i.e. when $g=h=1$) admits Anosov diffeomorphisms. However, the case of  $\Sigma_g\times\Sigma_h$, where at least one of $g$ or $h$ is  $\geq2$, seems to be more subtle:

\begin{prob}\normalfont{(Gogolev-Lafont \cite[Section 7.2]{GL}).}\label{p:GL}
Does the product of two closed aspherical surfaces at least one of which is hyperbolic admit an Anosov diffeomorphism?
\end{prob}

\subsection*{Outline}
In Section \ref{s:thurston} we enumerate the Thurston geometries in dimensions up to four and gather some preliminaries. 
In Sections \ref{s:hyperbolic}, \ref{s:solvandcomp} and  \ref{s:products} we prove Theorem \ref{t:main}.

\subsection*{Acknowledgments}
Parts of this project were carried out during research stays at CUNY Graduate Center and at IH\'ES in 2019. I am grateful to Dennis Sullivan and to Misha Gromov respectively for their hospitality. Also, I would like to thank Morris Hirsch for useful correspondence, as well as an anonymous referee for constructive comments. The support by the Swiss NSF, under grant FNS200021$\_$169685, is also gratefully acknowledged.

\section{Thurston geometries and finite covers}\label{s:thurston}

We begin our discussion by recalling the classification of the Thurston geometries in dimension four, as well as some simple general facts about Anosov diffeomorphisms and of their finite covers.

\medskip

Let $\mathbb{X}^n$ be a complete simply connected $n$-dimensional Riemannian manifold.  A closed manifold $M$ {\em carries the $\mathbb{X}^n$ geometry} or it is an {\em $\mathbb{X}^n$-manifold} in the sense of Thurston, if it is diffeomorphic to a quotient of
$\mathbb{X}^n$ by a lattice $\Gamma$ (the fundamental group of $M$) in the group of isometries $\mathrm{Isom}(\mathbb{X}^n)$ (acting effectively and transitively). We say that two geometries $\mathbb{X}^n$ and $\mathbb{Y}^n$ are the same if there exists a diffeomorphism $\psi \colon \mathbb{X}^n
\to \mathbb{Y}^n$ and an isomorphism $\mathrm{Isom}(\mathbb{X}^n) \to \mathrm{Isom}(\mathbb{Y}^n)$ which maps each element $g \in \mathrm{Isom}(\mathbb{X}^n)$ to $\psi \circ g \circ \psi^{-1} \in \mathrm{Isom}(\mathbb{Y}^n)$.

In dimension one, the circle is the only closed manifold and it is a quotient of the real line $\R$ by $\Z$. In dimension two, a closed surface carries one of the geometries $S^2$, $\R^2$ or $\mathbb{H}^2$ and (virtually) it is respectively $S^2$, $T^2$  or a hyperbolic surface $\Sigma_g$ (of genus $g\geq2$). In dimension three, Thurston~\cite{Th1} proved that there exist eight homotopically unique geometries, namely $\mathbb{H}^3$, $Sol^3$,
$\widetilde{SL_2}$, $\mathbb{H}^2 \times \R$, $Nil^3$, $\R^3$, $S^2 \times \R$ and $S^3$. In Table \ref{table:3geom}, we list the finite covers for manifolds in each of those geometries (see~\cite{Th1,Scott:3-mfds,Agol}), as we will use several of those properties in our proofs.

\begin{table}
\centering
{\small
\begin{tabular}{r|l}
Geometry $\mathbb{X}^3$ & $M$ is finitely covered by...\\
\hline
$\mathbb{H}^3$     & a mapping torus of a hyperbolic surface with pseudo-Anosov monodromy\\
$Sol^3$            & a mapping torus of the 2-torus $T^2$ with hyperbolic monodromy\\
$\widetilde{SL_2}$ & a non-trivial circle bundle over a hyperbolic surface\\  
$Nil^3$            & a non-trivial circle bundle over $T^2$\\
$\mathbb{H}^2 \times \R$ & a product of the circle with a hyperbolic surface\\
$\R^3$             & the $3$-torus $T^3$\\
$S^2 \times \R$    & the product $S^2 \times S^1$\\
$S^3$              & the $3$-sphere $S^3$
\end{tabular}}
\newline
\caption{{\small Finite covers of Thurston geometric closed 3-manifolds.}}\label{table:3geom}
\end{table}

The 4-dimensional geometries were classified by Filipkiewicz in his thesis~\cite{Filipkiewicz}. According to that classification, there are
eighteen geometries with compact representatives, and an additional geometry which is not realizable by a compact $4$-manifold. The list with the eighteen geometries is given in Table \ref{table:4geom}, and it is arranged so that it serves as an organising principle for the forthcoming sections. (Note that nineteen geometries appear, because $Sol^3\times\R$ is the geometry $Sol^4_{m,n}$ when $m=n$.) The individual characteristics of each geometry needed for our proofs will be given when dealing with each geometry. As pointed out in the introduction, among the most mysterious geometries with respect to Anosov diffeomorphisms is $\mathbb{H}^2\times\mathbb{H}^2$. Manifolds modeled on this geometry are divided into the ``reducible" and ``irreducible" ones, and different phenomena occur depending on where they belong. 

\begin{table}[!ht]
\centering
{\small
\begin{tabular}{r|l}
Type of the geometry & Geometry $\mathbb{X}^4$\\
\hline
Hyperbolic & $\mathbb{H}^4$, $\mathbb{H}^2(\mathbb{C})$\\        
                Solvable non-product   & $Nil^4$, 
$Sol^4_{m \neq n}$, $Sol^4_0$,
             $Sol^4_1$\\
             Compact non-product &  $S^4$, $\mathbb{CP}^2$\\
                 Product    & $\R^4$, $Nil^3\times\mathbb{R}$, $S^2\times S^2$, $S^2\times\mathbb{H}^2$, $S^2\times \R^2$,  $S^3 \times \R$, 
 $\mathbb{H}^3\times\mathbb{R}$,
          \\
       &   $\mathbb{H}^2\times\mathbb{R}^2$, $\mathbb{H}^2\times\mathbb{H}^2$,
             $Sol^3\times\R$,  $\widetilde{SL_2}\times\mathbb{R}$ \\
\end{tabular}}
\newline
\caption{{\small The 4-dimensional Thurston geometries with compact representatives.}}\label{table:4geom}
\end{table}

The virtual properties of geometric 4-manifolds will be used extensively in our study.
We thus end this preliminary section with the following general lemmas (see~\cite{GL} and~\cite{GH} respectively):

\begin{lem}
\label{l:pre1}
Let $M$ be a closed manifold and $p\colon\overline M\to M$ be a finite covering. If $f\colon M\to M$ is a diffeomorphism, then there is an $m\geq 0$ such that $f^m$ lifts to a diffeomorphism $\overline{f^m}\colon\overline M\to \overline M$, i.e. the following diagram commutes.
$$
\xymatrix{
\overline M\ar[d]_{p} \ar[r]^{\overline{f^m}}&  \ar[d]^{p} \overline M\\
M\ar[r]^{f^m}& M \\
}
$$
\end{lem}

\begin{lem}
\label{l:pre2}
If $f\colon M\to M$ is a transitive Anosov diffeomorphism and there is a lift $\overline f\colon\overline M\to\overline M$ of $f$ for some cover $\overline M$ of $M$, then $\overline f$ is transitive.
\end{lem}

\section{Hyperbolic geometries}\label{s:hyperbolic}

We now begin the proof of Theorem \ref{t:main}. We first deal with the hyperbolic geometries. 

\medskip

The real and complex hyperbolic geometries, $\mathbb{H}^4$ and $\mathbb{H}^2(\mathbb{C})$ respectively, are generally among the less understood of the eighteen geometries in dimension four. However, the machinery developed for hyperbolic manifolds in general suffices to rule out Anosov diffeomorphisms on 4-manifolds carrying one of those geometries.
The following theorem is now well-known to experts, but nevertheless we give a proof for the sake of completeness and in order to include some useful facts about Anosov diffeomorphisms which will be used below as well, such as properties of their Lefschetz numbers.

\begin{thm}[\cite{Yano,GL}]\label{t:finiteout}
If $M$ is a negatively curved manifold, then $M$ does not support Anosov diffeomorphisms.
\end{thm}
\begin{proof}
The first proof due to Yano~\cite{Yano} rules out the existence of transitive Anosov diffeomoprhisms. Let $M$ be negatively curved and suppose $f\colon M\to M$ is a transitive Anosov diffeomorphism. Since codimension one Anosov diffeomorphisms exist only on tori~\cite{Fr,Newhouse}, we can clearly assume that the dimension of $M$ is at least four and the codimension $k$ of $f$ is at least two. By Ruelle-Sullivan~\cite{RS}, the transitivity assumption implies the existence of a homology class $a\in H_l(M;\R)$ such that $f_*(a)=\lambda\cdot a$ for some $\lambda>1$, where $l=k>1$ or $l=\dim(M)-k>1$. This means that the simplicial $\ell^1$-semi-norm of $a$ is zero which is impossible because $M$ is negatively curved~\cite{Gromov,IY}.

An argument that rules out the existence of any Anosov diffeomorphism on a negatively curved manifold $M$ of dimension $\geq 3$ was given by Gogolev-Lafont~\cite{GL}, using the fact that the outer automorphism group $\mathrm{Out}(\pi_1(M))$ is finite (the latter can be derived by combining results of Paulin~\cite{Pau}, Bestvina-Feighn~\cite{BF} and Bowditch~\cite{Bow}; see~\cite[Corollary 4.5]{GL}). The finiteness of $\mathrm{Out}(\pi_1(M))$ and the asphericity of $M$ (being negatively curved) implies that an iterate $f^l$ of (a finite covering of) $f$ induces the identity on cohomology. (One already concludes that $M$ does not support transitive Anosov diffeomorphisms by Ruelle-Sullivan~\cite{RS} or Shiraiwa~\cite{Shi}.)  Thus the Lefschetz numbers $\Lambda$ (i.e. the sum of indices of the fixed points) of all powers of $f^l$ are uniformly bounded, which is in contrast with the growth of periodic points of $f^l$, because of the equation
\begin{equation}\label{eq.FixAnosov}
|\Lambda(f^{m})|=|\mathrm{Fix}(f^{m})| = re^{mh_{top}(f)} + o(e^{mh_{top}(f)}), \ m\geq 1,
\end{equation}
where $h_{top}(f)$ is the topological entropy of $f$ and $r$ is the number of transitive basic sets with entropy equal to $h_{top}(f)$; see~\cite[Lemma 4.1]{GL} for details.
\end{proof}

We immediately obtain:

\begin{cor}
Closed 4-manifolds modeled on the geometry $\mathbb{H}^4$ or $\mathbb{H}^2(\mathbb{C})$ do not support Anosov diffeomorphisms.
\end{cor}

\begin{rem}
As observed in~\cite{GL}, the finiteness of the outer automorphism group of the fundamental group of every negatively curved manifold of dimension $\geq 3$ caries over the outer automorphism group of the fundamental group of a finite product $M_1\times\cdots\times M_s$ of negatively curved manifolds $M_i$ of dimensions $\geq3$. Thus $M_1\times\cdots\times M_s$ does not support Anosov diffeomorphisms. However, this obstruction does not apply anymore if one of the $M_i$ is 2-dimensional, i.e. a hyperbolic surface. In~\cite[Theorem 1.4 and Example 4.3]{NeoAnosov1} we ruled out Anosov diffeomorphisms on products of a hyperbolic surface with certain higher dimensional negatively curved manifolds. It seems that an alternative method is required in general in order to rule out Anosov diffeomorphisms on product of two surfaces at least one of which is hyperbolic (those manifolds correspond to the geometry $\mathbb{H}^2\times\R^2$ or the reducible $\mathbb{H}^2\times\mathbb{H}^2$ geometry); cf. Problem \ref{p:GL} and \cite[Section 7.2]{GL} for further discussion.
\end{rem}

\section{Non-product, solvable and compact geometries}\label{s:solvandcomp}

In this section, we deal with the geometries $Nil^4$, $Sol^4_{m \neq n}$, $Sol^4_0$, $Sol^4_1$, $S^4$ and $\CP^2$. 

\subsection{Solvable non-product geometries}

\subsubsection{The geometry $Nil^4$.}\label{ss:Nil}

Let $M$ be a closed 4-manifold modeled on the geometry $Nil^4$. Then (a finite index subgroup of) the fundamental group of $M$ has a presentation
\[
 \pi_1(M) = \langle x,y,z,t \ \vert \ txt^{-1}=x, \ tyt^{-1}=x^kyz^l, \ tzt^{-1} = z, [x,y]=z, \ xz=zx, \ yz=zy \rangle,
\]
$k\geq 1$, $l\in\Z$, with center $C(\pi_1(M)) = \langle z \rangle$. The quotient of $\pi_1(M)$ by its center is given by
\[
 \pi_1(M)/\langle z\rangle = \langle x,y,t \ \vert \ [t,y]=x^k, \ xt=tx, \ xy=yx \rangle;
\]
see~\cite[Prop. 6.10]{NeoIIPP} and~\cite[Section 8.7]{Hil} for details. We moreover observe that $\pi_1(M)$ is an extension $\Z^3\rtimes_\theta\Z=\langle z,x,t\rangle\rtimes_\theta\langle y \rangle$, where the automorphism $\theta\colon\Z^3\to\Z^3$ is given by
\[
\left(\begin{array}{ccc}
   1 & -1 & -l \\
   0 & 1 & -k \\
   0 & 0 & 1 \\
\end{array} \right).
\]
Let $f\colon M\to M$ be a diffeomorphism. Then $f_\sharp\colon\pi_1(M)\to\pi_1(M)$ induces an automorphism of $\pi_1(M)/\langle z\rangle$, because $f_\sharp(\langle z\rangle)=\langle z\rangle$.  Since $C( \pi_1(M)/\langle z\rangle)=\langle x\rangle$, we deduce that $f_\sharp(x)=z^nx^m$, for some $n,m\in\Z$, $m\neq0$. Now, the relation $txt^{-1}=x$ is mapped to $f_\sharp(t)x^mf_\sharp(t)^{-1}=x^m$, thus, by $[x,y]=z$, the image $f_\sharp(t)$ does not contain any powers of $y$. Combining all together, 
we conclude, using the commutative diagram
$$
\xymatrix{
\pi_1(M)\ar[d]_{h} \ar[r]^{{f}_\sharp}&  \ar[d]^{h}  \pi_1(M)\\
H_1(M;\Z) \ar[r]^{{f}_*}& H_1(M;\Z), \\
}
$$
where $h\colon\pi_1(M)\to H_1(M;\Z)=\pi_1(M)/[\pi_1(M),\pi_1(M)]$ denotes the Hurewicz homomorphism, that the induced isomorphism in homology $f_*$ maps $\bar t\in H_1(M;\Z)/\Tor H_1(M;\Z)$ to a multiple of itself. The induced automorphism on $H_1(M;\Z)/\Tor H_1(M;\Z)=\langle \bar t\rangle\times \langle \bar y\rangle=\Z\times\Z$ implies in fact that $f_*(\bar t)=\bar t$ and thus $f$ cannot be Anosov by Lemma \ref{l:pre1} and the following result of Hirsch:

\begin{thm}{\normalfont(\cite[Theorem 1]{Hirsch}).}\label{t:Hirschcov}
Let $f\colon M\to M$ be an Anosov diffeomorphism and a non-trivial cohomology class $u\in H^1(M;\Z)$ such that $(f^*)^m(u)=u$, for some positive integer $m$. Then the infinite cyclic covering of $M$ corresponding to $u$ has infinite dimensional rational homology.
\end{thm}

\begin{rem}
The infinite cyclic covering of $M$ corresponding to $u$ is the covering whose fundamental group is given by the kernel of the composition
\[
\pi_1(M)\stackrel{h}\longrightarrow H_1(M)\xrightarrow{<u,\cdot>}\Z,
\]
where $h$ is the Hurewitz homomorphism as above and $<u,\cdot>$ the Kronecker product. Note that
Hirsch's result amounts again to the fact that finite dimensional rational homology of the above infinite cyclic covering would imply vanishing of the Lefschetz number of (an iterate of) $f$, which is impossible for an Anosov diffeomorphism. 
\end{rem}

\begin{rem}
As we conclude from our proof, the induced automorphism
\[
f_*\colon H_1(M;\R)\to H_1(M;\R)
\] 
has a root of unity as eigenvalue. Then~\cite[ Corollary 2]{Hirsch} implies that $f$ is not Anosov (as an application of Theorem \ref{t:Hirschcov}).
For a manifold $M$ with polycyclic fundamental group and whose universal covering has finite dimensional rational homology,~\cite[Theorem 4]{Hirsch} tells us that a diffeomorphism $f\colon M\to M$ is not Anosov if there is a root of unity among the eigenvalues of $f_*\colon H_1(M;\R)\to H_1(M;\R)$. Also, note that
~\cite{Mal} determines which nilpotent manifolds admit Anosov diffeomorphism up to dimension six,  hence also covers the case of the $Nil^4$ geometry. In our proof we did not (explicitly) use the fact that $\pi_1(M)$ is polycyclic, but we rather exhibited a cohomology class satisfying Theorem \ref{t:Hirschcov}.
\end{rem}

\subsubsection{The geometries $Sol_{m \neq n}^4$, $Sol^4_0$ and $Sol^4_1$}

For the geometries $Sol^4_{m \neq n}$, $Sol^4_0$ and $Sol^4_1$ a weaker statement (Theorem \ref{t:Hirsch} below) than that of Theorem \ref{t:Hirschcov}, based on the first Betti number, suffices to rule out Anosov diffeomorphisms. We begin by recalling the model spaces of those geometries: 

\medskip

Suppose $m$ and $n$ are positive integers, $a > b > c$ reals such that $a+b+c=0$ and $e^a,e^b,e^c$ are
roots of the polynomial $P_{m,n}(\lambda)=\lambda^3-m\lambda^2+n\lambda-1$. For $m \neq n$, the Lie group $Sol_{m \neq n}^4$ is defined as a semi-direct product $\R^3 \rtimes
\R$, where $\R$ acts on $\R^3$ by
\[
t \mapsto 
\left(\begin{array}{ccc}
   e^{at} & 0 & 0 \\
   0 & e^{bt} & 0 \\
   0 & 0 & e^{ct} \\
\end{array} \right).
\]
Note that the case $m=n$ gives $b = 0$ and corresponds to the product geometry $Sol^3 \times \R$.

\medskip

If two roots of the polynomial $P_{m,n}$ are required to be equal, then we obtain the model space of the $Sol_0^4$ geometry, again defined as a semi-direct product $\R^3
\rtimes \R$, where now the action of $\R$ on $\R^3$ is given by
\[
t \mapsto 
\left(\begin{array}{ccc}
   e^{t} & 0 & 0 \\
   0 & e^{t} & 0 \\
   0 & 0 & e^{-2t} \\
\end{array} \right).
\]

Closed manifolds modeled on the geometries $Sol_{m \neq n}^4$ and $Sol_0^4$ have the following property:

\begin{thm}[\normalfont{\cite[Corollary 8.5.1]{Hil}}]\label{t:mappingtorisolvable1} 
 Every closed manifold carrying one of the geometries $Sol_0^4$ or $Sol_{m \neq n}^4$ is a mapping torus of a hyperbolic automorphism of the 3-torus. 
\end{thm}

Finally, the Lie group $Sol_1^4$ is defined as a semi-direct product $Nil^3 \rtimes \R$, where $\R$ acts on the 3-dimensional Heisenberg group
\[
 Nil^3 = 
\Biggl\{ \left( \begin{array}{ccc}
  1 & x & z \\
  0 & 1 & y \\
  0 & 0 & 1 \\
\end{array} \right) \biggl\vert
\ x,y,z \in \R \Biggl\}
\]
by
\[
t \mapsto 
\left(\begin{array}{ccc}
   1 & e^{-t}x & z \\
   0 & 1 & e^{t}y \\
   0 & 0 & 1 \\
\end{array} \right).
\]

Closed manifolds modeled on the geometry $Sol_1^4$ can be described as follows:

\begin{thm}[\normalfont{\cite[Theorem 8.9]{Hil}}]\label{t:mappingtorisolvable2} 
A closed oriented manifold carrying the geometry $Sol_1^4$ is a mapping torus of a self-homeomorphism of a $Nil^3$-manifold.
\end{thm}

Using this, one can moreover derive that every closed $Sol_1^4$-manifold is a virtually non-trivial circle bundle over a $Sol^3$-manifold~\cite[Prop. 6.15]{NeoIIPP}.

\medskip

The descriptions of the fundamental groups of manifolds carrying one of the above solvable geometries suffice to exclude Anosov diffeomorphisms on them by the following result of Hirsch, which is a consequence of the more general Theorem \ref{t:Hirschcov}:

\begin{thm}{\normalfont(\cite[Theorem 8]{Hirsch}).}\label{t:Hirsch}
Suppose $M$ is a compact manifold such that 
\begin{itemize}
\item[(a)] $\pi_1(M)$ is virtually polycyclic;
\item[(b)] the universal covering of $M$ has finite dimensional rational homology;
\item[(c)] $H^1(M;\Z)\cong\Z$.
\end{itemize}
Then $M$ does not support Anosov diffeomorphisms.
\end{thm}

\begin{cor}
Closed 4-manifolds modeled on one of the geometries $Sol_0^4$, $Sol_{m \neq n}^4$ or $Sol_1^4$ do not support Anosov diffeomorphisms.
\end{cor}
\begin{proof}
After passing to a finite covering we may assume that $M$ is oriented. 

If $M$ carries one of the geometries $Sol_0^4$ or $Sol_{m \neq n}^4$, then by Theorem \ref{t:mappingtorisolvable1}
 \[
 \pi_1(M)\cong \pi_1(T^3) \rtimes_{\theta_M} \langle t \rangle,
 \]
where $\pi_1(T^3) = \Z^3 = \langle x_1,x_2,x_3 \vert \ [x_i,x_j] = 1 \rangle$ and the automorphism $\theta_M \colon \Z^3\to \Z^3$ is hyperbolic. Thus, $H^1(M;\Z)\cong\Z$, and since $M$ is aspherical and $\pi_1(M)$ polycyclic, Theorem \ref{t:Hirsch} and Lemma \ref{l:pre1} tell us that $M$ cannot support Anosov diffeomorphisms.

If $M$ carries the geometry $Sol_1^4$, then by Theorem \ref{t:mappingtorisolvable2} (see also~\cite[Prop. 6.15]{NeoIIPP}) a presentation of its fundamental group is given by
\begin{eqnarray*}
  \pi_1(M) = &\langle x,y,z,t \ \vert & txt^{-1}=x^ay^cz^k, \ tyt^{-1}=x^by^dz^l, \ tzt^{-1} =z,\\
             &\ & 
             [x,y]=z, \ xz=zx, \ yz=zy \rangle,
\end{eqnarray*}
where $k,l\in\Z$ and the matrix 
\[
\left(\begin{array}{cc}
   a & b \\
   c & d \\
\end{array} \right)\in \SL_2(\Z)
\]
 has no roots of unity. 
The abelianization of $\pi_1(M)$ implies $H^1(M;\Z)\cong\Z$. Since moreover $M$ is aspherical and $\pi_1(M)$ is polycyclic, we deduce by Theorem \ref{t:Hirsch} and Lemma \ref{l:pre1} that $M$ does not support Anosov diffeomorphisms.
\end{proof}

\begin{rem}
Note that Theorem \ref{t:Hirsch} is not applicable to a $Nil^4$ manifold $M$ (cf. Section \ref{ss:Nil}), because $H^1(M;\Z)\cong\Z^2$.
\end{rem}

\subsection{Compact non-product geometries}

Among the simplest cases are the compact geometries $S^4$ and $\CP^2$. 

\subsubsection{The geometry $S^4$}
The only closed oriented 4-manifold modeled on $S^4$ is $S^4$ itself~\cite[Section 12.1]{Hil}. Clearly, any orientation preserving diffeomorphism $f$ of $S^4$ induces the identity on $H^*(S^4)$, and as we have seen this makes it impossible for $f$ to be Anosov (cf. equation \ref{eq.FixAnosov}).

\subsubsection{The geometry $\CP^2$}
As for the geometry $S^4$, the only closed oriented 4-manifold modeled on $\CP^2$ is $\CP^2$ itself~\cite[Section 12.1]{Hil}. Suppose 
\[
f\colon \CP^2\to \CP^2
\]
 is a diffeomorphism. The cohomology groups of $\CP^2$ are $\Z$ in degrees 0, 2 and 4 and trivial otherwise. So, after possibly passing to an iterate of $f$, we observe, by the naturality of the cup product, that $f$ must induce the identity on cohomology. Thus $f$ cannot be Anosov.
 
\section{Product geometries}\label{s:products}

In order to complete the proof of Theorem \ref{t:main}, we need to examine the product geometries that are not excluded by the statement of Theorem \ref{t:main}, i.e. the geometries  $\mathbb{H}^3\times\mathbb{R}$, $Sol^3\times\R$, 
 $\widetilde{SL_2}\times\mathbb{R}$, $Nil^3\times\R$, the irreducible $\mathbb{H}^2\times\mathbb{H}^2$ geometry, 
 $S^2\times\mathbb{H}^2 $, $S^2\times \R^2$, $S^3 \times \R$ and $S^2\times S^2$.
 
\medskip

\subsection{Products with a compact factor} 

\subsubsection{The geometry $S^2\times S^2$}

The question of whether $S^2\times S^2$ supports Anosov diffeomorphisms was asked by Ghys in the 1990's and, although it has a quite straightforward solution using the intersection form, was only recently answered by Gogolev and Rodriguez Hertz~\cite{GH}. Suppose $f\colon S^2\times S^2\to S^2\times S^2$ is a diffeomorphism (or, more generally, a map of degree $\pm1$). The K\"unneth formula gives
\[
H^2(S^2\times S^2)=(H^2(S^2)\otimes H^0(S^2))\oplus(H^0(S^2)\otimes H^2(S^2)).
\]
Let $\omega_{S^2}\times 1\in H^2(S^2)\otimes H^0(S^2)$ and $1\times\omega_{S^2}\in H^0(S^2)\otimes H^2(S^2)$ be the corresponding cohomological fundamental classes. After possibly replacing $f$ by $f^2$, we can assume that $\deg(f)=1$. The effect of $f$ on the above classes is given by
\[
f^*(\omega_{S^2}\times1)=a\cdot(\omega_{S^2}\times 1)+b\cdot(1\times\omega_{S^2}), \ a,b\in\Z,
\]
and
\[
f^*(1\times\omega_{S^2})=c\cdot(\omega_{S^2}\times 1)+d\cdot(1\times\omega_{S^2}), \ c,d\in\Z.
\]
Thus, by the naturality of the cup product we obtain
\begin{equation}\label{eq.S2}
ad+bc=1.
\end{equation}
Also, since the cup product of $\omega_{S^2}\times 1$ with itself vanishes, we obtain
\[
0=f^*((\omega_{S^2}\times1)\cup(\omega_{S^2}\times1))=f^*(\omega_{S^2}\times1)\cup f^*(\omega_{S^2}\times1)=2ab\cdot (\omega_{S^2\times S^2}),
\]
and so
\begin{equation}\label{eq.S2b}
ab=0.
\end{equation}
Similarly, since $(1\times\omega_{S^2})\cup(1\times\omega_{S^2})=0$, we obtain
\begin{equation}\label{eq.S2c}
cd=0.
\end{equation}
If $a=0$, then (\ref{eq.S2}), (\ref{eq.S2b}) and (\ref{eq.S2c}) imply $b=c=\pm 1$ and $d=0$. If $b=0$, then again by the same equations we obtain $a=d=\pm1$ and $c=0$. Thus, after possibly replacing $f$ by $f^2$, we deduce that $f$ induces the identity in cohomology. Therefore, the Lefschetz numbers of all powers of $f$ are uniformly bounded, and so $f$ cannot be Anosov diffeomorphism (cf. equation \ref{eq.FixAnosov}).

\begin{rem}
Alternatively to the above argument, note that, since $f$ is a diffeomorphism, the matrix for the induced action on $H^2$ lies in $\GL_2(\Z)$, hence $ad-bc = \pm1$. Combining this with equation (\ref{eq.S2}), we can find the two possible integer solutions as above.
\end{rem}

\subsubsection{The geometry $S^2\times \R^2$}

In that case, $M$ is (finitely covered by) $S^2\times T^2$~\cite[Theorem 10.10]{Hil}. Since every map $S^2\to T^2$ has degree zero, if $f\colon S^2\times T^2\to S^2\times T^2$ is a 
diffeomorphism, then the effect of $f$ on the cohomological fundamental classes $\omega_{S^2}\times 1\in H^2(S^2)\otimes H^0(T^2)$ and $1\times\omega_{T^2}\in H^0(S^2)\otimes H^2(T^2)$ is given by
\[
f^*(\omega_{S^2}\times1)=a\cdot(\omega_{S^2}\times 1)+b\cdot(1\times\omega_{T^2}), \ a,b\in\Z,
\]
and
\[
f^*(1\times\omega_{T^2})=d\cdot(1\times\omega_{T^2}), \ d\in\Z;
\]
see~\cite{Neodegrees} for details.
As before, we assume that $\deg(f)=1$, and so the naturality of the cup product yields
\begin{equation}\label{eq.S2mixied}
ad=1.
\end{equation}
In particular, $a=d=\pm1$. Also, $b=0$ by the vanishing of the cup product of $\omega_{S^2}\times1$ with itself. 

Recall that, by Franks~\cite{Fr} and Newhouse~\cite{Newhouse}, if a manifold admits a codimension one Anosov diffeomorphism, then it must be homeomorphic to a torus. Thus, 
 if $f$ is Anosov, then we may assume that it has codimension two. In that case, Ruelle-Sullivan's work~\cite{RS} gives us a class $\alpha\in H^2(S^2\times T^2;\R)$ such that $f^*(\alpha)=\lambda\cdot\alpha$ for some positive real $\lambda\neq1$. We have 
\[
\alpha=\xi_1\cdot(\omega_{S^2}\times1)+\xi_2\cdot(1\times\omega_{T^2}), \ \xi_1,\xi_2\in\R,
\]
and so $f^*(\alpha)=\lambda\cdot\alpha$ yields
\begin{equation}\label{eq.S2mixed-3}
\lambda\xi_1=a\xi_1=\pm\xi_1
\end{equation}
and
\begin{equation}\label{eq.S2mixed-4}
\lambda\xi_2=d\xi_2=\pm\xi_2.
\end{equation}
If $\xi_1\neq0$, then (\ref{eq.S2mixed-3}) becomes $\lambda=\pm1$, which is impossible. If $\xi_1=0$, then $\xi_2\neq0$ and (\ref{eq.S2mixed-4}) yields again the absurd conclusion $\lambda=\pm1$.

This shows that $S^2\times T^2$ does not support transitive Anosov diffeomorphisms.

\subsubsection{The geometry $S^2\times\mathbb{H}^2$}

If $M$ is modeled on the geometry $S^2\times\mathbb{H}^2$, then $M$ is virtually an $S^2$-bundle over a closed hyperbolic surface $\Sigma_h$~\cite[Theorem 10.7]{Hil}. The case of $S^2\times\Sigma_h$ can be treated using the same argument as for $S^2\times T^2$. More generally, Gogolev-Rodriguez Hertz showed that a fiber bundle $S^{2n}\to E\to B$, where $B$ is $2n$-dimensional, does not support transitive Anosov diffeomorphisms~\cite[Theorem 1.1]{GH}, which covers as well the geometry $S^2\times \R^2$. Their argument uses again equation \ref{eq.FixAnosov} and cup products via the Gysin sequence 
\[
0\longrightarrow H^{2n}(B;\Z)\longrightarrow H^{2n}(E;\Z)\longrightarrow H^0(B;\Z)\longrightarrow0.
\]
Note that in our case, $2n=2$ is the only case of interest for the codimension; we refer to~\cite{GH} for the complete argument.

\subsubsection{The geometry $S^3\times\R$}

A closed 4-manifold modeled on the geometry $S^3\times\R$ is virtually a product $S^3\times S^1$~\cite[Ch. 11]{Hil}, which clearly does not support Anosov diffeomorphisms because $H_2(S^3\times S^1)=0$ and $H_1(S^3\times S^1)=\Z$.

\subsection{The irreducible $\mathbb{H}^2\times\mathbb{H}^2$ geometry}

As for the hyperbolic geometries, if $M$ is an irreducible manifold modeled on the geometry $\mathbb{H}^2\times\mathbb{H}^2$, then $\pi_1(M)$ has finite outer automorphism group by the strong rigidity of Mostow, Prasad and Margulis. Thus the proof of Theorem \ref{t:finiteout} implies that $M$ does not support Anosov diffeomorphisms.

\subsection{Aspherical products with a circle factor}

Finally, we deal with the product geometries $\mathbb{H}^3\times\R$, $Sol^3\times\R$, $\widetilde{SL_2}\times\R$ and $Nil^3\times\R$. 

\subsubsection{The geometries $\widetilde{SL_2}\times\R$ and $Nil^3\times\R$}\label{ss:Hirsch2}

Let $M$ be a closed 4-manifold modeled on the geometry $\widetilde{SL_2}\times\R$ or the geometry $Nil^3\times\R$. Then $M$ is finitely covered by a product $N\times S^1$, where $N$ is an $\widetilde{SL_2}$-manifold or a $Nil^3$-manifold respectively~\cite{Hil}. We can moreover assume that $N$ is a non-trivial circle bundle over a surface $\Sigma_g$ of genus $g$, where $g\geq2$ if $N$ is an $\widetilde{SL_2}$-manifold and $g=1$ if $N$ is a $Nil^3$-manifold; cf. Table \ref{table:3geom}. In particular, the center of $\pi_1(N\times S^1)$ has rank two. Since (a finite power of) the generator of the fiber of $N$ vanishes in $H_1(N)$, we deduce that, for any diffeomorphism $f\colon N\times S^1\to N\times S^1$ the generator of $H_1(S^1)$ maps to a power of itself (modulo torsion). That is, in cohomology
\[
f^*(1\times\omega_{S^1})=a\cdot(1\times\omega_{S^1}), \ a\in\Z.
\]
Moreover, because $N$ does not admit maps of non-zero degree from direct products~\cite{KN} and the degree three cohomology of $N\times S^1$ is
\[
H^3(N\times S^1)\cong H^3(N)\oplus(H^2(N)\otimes H^1(S^1)),
\]
we obtain
\[
f^*(\omega_N\times1)=b\cdot(\omega_N\times1), \ b\in\Z;
\]
see~\cite[Proof of Theorem 1.4]{Neodegrees} for further details. 
Since $\deg(f)=\pm1$, we deduce that $a,b\in\{\pm1\}$. Thus, after possibly replacing $f$ by $f^2$, we may assume that 
\[
f^*(1\times\omega_{S^1})=1\times\omega_{S^1}.
\]
Now Theorem \ref{t:Hirschcov} and Lemma \ref{l:pre1} imply that $f$ cannot be Anosov.

Alternatively, since the generator of $H_1(S^1)$ maps to (a power of) itself, we can conclude that $f$ is not Anosov by~\cite[Corollary 2]{Hirsch}, again as an application of Theorem \ref{t:Hirschcov}.

\begin{rem}\label{r:Hirsch}
An example of a $Nil^3$ manifold is given by the mapping torus $M_A$ of $T^2$ with monodromy
\[
A=\left(\begin{array}{cc}
   1 & 1\\
   0 & 1\\
 \end{array} \right).
\]
 As we have seen above, $M_A\times S^1$ does not support Anosov diffeomorphisms. Now, 
clearly $A^m\neq I_2=\left(\begin{array}{cc}
   1 & 0\\
   0 & 1\\
 \end{array} \right)$ for all $m\neq 0$ and, moreover,
\[
\pi_1(M_A)=\langle x,y,z \ | \ [x,y]=z,\ xz=zx, \ yz=zy\rangle,
\]
which has non-trivial center $C(\pi_1(M_A))=\langle z \rangle$. Therefore, in the proof of~\cite[Theorem 9(a)]{Hirsch} -- which asserts that for any monodromy $A\colon T^n\to T^n$ such that $A^m\neq I_n$ for all $m\neq 0$, the product $M_A\times S^1$ does not support Anosov diffeomorphisms -- the claim that the generator of $H_1(S^1)$ maps to a power of itself is derived by the invalid conclusion that $C(\pi_1(M_A))$ is trivial. (We remark that this error does not affect the aforementioned Theorems \ref{t:Hirschcov} and \ref{t:Hirsch} from the same paper.)
\end{rem}

\subsubsection{The geometries $\mathbb{H}^3\times\R$ and $Sol^3\times\R$}

A closed 4-manifold $M$ modeled on the geometry $\mathbb{H}^3\times\R$ or the geometry $Sol^3\times\R$ is virtually a product $N\times S^1$, where $N$ is a hyperbolic 3-manifold or a $Sol^3$-manifold respectively~\cite{Hil}. In particular, the fundamental group $\pi_1(N\times S^1)$ has infinite cyclic center generated by the circle factor~\cite{Scott:3-mfds}; let us denote this by $\pi_1(S^1)=\langle z\rangle$.

Suppose $f\colon N\times S^1\to N\times S^1$ is a diffeomorphism. Then $f_\sharp(\langle z\rangle)=\langle z\rangle$, and therefore $f_*(\omega_{S^1})=\omega_{S^1}$ (up to taking $f^2$ if necessary) as in the above subsection (because $N$ does not admit maps of non-zero degree from direct products~\cite{KN}) or alternatively because the center and the commutator of $\pi_1(N\times S^1)$ intersect trivially. We deduce that $f$ cannot be Anosov by Theorem \ref{t:Hirschcov} and Lemma \ref{l:pre1}.

Alternatively for the case of hyperbolic $N$, the main result of~\cite{GL} implies that $N\times S^1$ does not support Anosov diffeomorphisms, because $\mathrm{Out}(\pi_1(N))$ is finite and $\pi_1(N)$ is Hopfian and has trivial intersection of maximal nilpotent subgroups. In fact, as shown in~\cite{NeoAnosov2}, the only properties needed to exclude Anosov diffeomorphisms on $N\times S^1$ is that $\mathrm{Out}(\pi_1(N))$ is finite and $\pi_1(N)$ has trivial center.

\medskip

The proof of Theorem \ref{t:main} is now complete.

\bibliographystyle{amsplain}

\end{document}